\documentclass[sn-basic]{sn-jnl}

\usepackage{tikz}
\usepackage[T1]{fontenc}
\usepackage[ruled,vlined]{algorithm2e}
\usepackage{mathtools}
\usepackage{latexsym}
\usepackage{indentfirst}
\usepackage{epsfig,enumerate,amsfonts,setspace,float}
\usepackage{wrapfig,lipsum}
\usepackage{lmodern}
\usepackage{anyfontsize}

\usepackage{natbib}

\usepackage{graphicx}%
\usepackage{multirow}%
\usepackage{amsmath,amssymb,amsfonts}%
\usepackage{amsthm}%
\usepackage[title]{appendix}%
\usepackage{xcolor}%
\usepackage{textcomp}%
\usepackage{manyfoot}%
\usepackage{booktabs}%

\usepackage{algpseudocode}%
\usepackage{listings}%

\usetikzlibrary{decorations.pathreplacing,automata,calc,positioning}
\usetikzlibrary{arrows}

\newcommand{\rhogr}{\rho_{\rm gr}}
\newcommand{\taugr}{\tau_{\rm gr}}

\newcommand{\vertex}{\node[vertex]}
\tikzstyle{vertex}=[circle, draw, inner sep=0pt, minimum size=6pt]

\newtheorem{theorem}{Theorem}%  meant for continuous numbers
\newtheorem{proposition}{Proposition}% 
\newtheorem{cl}{Claim}
\newtheorem{lemma}{Lemma}

\newtheorem{Ob}{Observation}

\begin{document}

\title[Article Title]{Computation of Grundy dominating sequences in (co-)bipartite graphs}


\author[1,2]{\fnm{Bo\v{s}tjan} \sur{Bre\v{s}ar}}\email{bostjan.bresar@um.si}

\author[3]{\fnm{Arti} \sur{ Pandey}}\email{arti@iitrpr.ac.in}

\author*[3]{\fnm{Gopika} \sur{Sharma}}\email{2017maz0007@iitrpr.ac.in}

\affil[1]{\orgdiv{Faculty of Natural Sciences and Mathematics}, \orgname{University of Maribor}, \country{Slovenia}}

\affil[2]{\orgdiv{Institute of Mathematics}, \orgname{Physics and Mechanics}, \orgaddress{Ljubljana}, \country{Slovenia}}

\affil[3]{\orgdiv{Department of Mathematics}, \orgname{Indian Institute of Technology Ropar}, \orgaddress{Rupnagar}, \country{India}}


\abstract{A sequence $S$ of vertices of a graph $G$ is called a dominating sequence of $G$ if $(i)$ each vertex $v$ of $S$ dominates a vertex of $G$ that was not dominated by any of the vertices preceding vertex $v$ in $S$, and $(ii)$ every vertex of $G$ is dominated by at least one vertex of $S$. The \textsc{Grundy Domination} problem is to find a longest dominating sequence for a given graph $G$. It has been known that the decision version of the \textsc{Grundy Domination} problem is NP-complete even when restricted to chordal graphs. In this paper, we prove that the decision version of the \textsc{Grundy Domination} problem is NP-complete for bipartite graphs and co-bipartite graphs. On the positive side, we present a linear-time algorithm that solves the \textsc{Grundy Domination} problem for chain graphs, which form a subclass of bipartite graphs.
}

\keywords{dominating sequence, bipartite graph, chain graphs, computational complexity, linear-time algorithm}

%\noindent
%\textbf{Mathematics Subject Classification} 05C69 \sep 05C65

\pacs[Mathematics Subject Classification]{05C69, 05C65, 05C85}

\maketitle

\section{Introduction}\label{sec1}
Graph domination is an established area of graph theory with an extremely rich literature and a number of applications. Given a graph $G=(V,E)$ a set $S \subseteq V$ is \emph{dominating} if every vertex $x \in V\setminus S$ is adjacent to a vertex in $S$. A dominating set of $G$ with minimum cardinality is called a \emph{minimum dominating set} of $G$. The cardinality of a minimum dominating set of $G$ is the \emph{domination number} of $G$ and is denoted by $\gamma(G)$. The \textsc{Minimum Domination} problem is to find a dominating set of cardinality $\gamma(G)$. The \textsc{Minimum Domination} problem and its variations have numerous applications in real world problems including social networks, facility location problems, routing problems. An extensive overview on graph domination can be found in three recent monographs~\citep{dom1,dom2,dom5}.

\medskip
The domination game was introduced in~\citep{game1} using the approach of building a dominating set of a graph one vertex at a time. The game is played on a graph $G$ by two players, named Dominator and Staller, who are building a dominating set of $G$. These two players have the opposite goals in the game: one wants the game to end in as few moves as possible while the other one wants to extend the length of the game. They are taking turns in which a player chooses a vertex that has to dominate a vertex that has not been dominated by any of the previously chosen vertices. The game ends  when there is no more vertex to choose, that is, the set of vertices chosen during the game is a dominating set. Thus, during the game a sequence of vertices is selected. In several papers, authors considered a sequence obtained by the same basic rule, yet assuming that only the slow player plays the game. This leads to the following definitions.

\medskip
Let $S=(v_1,v_2, \ldots,v_k)$ be a sequence of vertices of $G$, and let $\widehat{S}=\{v_1,\ldots,v_k\}$. 
A sequence $S=(v_1, v_2, \ldots, v_k)$ is a {\em closed neighborhood sequence} if $N[v_i] \setminus \bigcup_{j=1}^{i-1}N[v_j] \ne\emptyset,$
holds for every $i\in\{2, 3, \ldots,k\}$. If, in addition, $\widehat{S}$ is a dominating set of $G$, then $S$ is a \emph{dominating sequence} of $G$. 
Clearly, the length $k$ of $S$ is bounded from below \textcolor{black}{by} $\gamma(G)$. A dominating sequence of maximum length in $G$ is a \emph{Grundy dominating sequence} (or, {\em GD-sequence} for short) of $G$. The cardinality of such a sequence is called the \emph{Grundy domination number} of $G$ and is denoted by $\gamma_{gr}(G)$.

\medskip
These concepts were introduced and studied in 2014 by Bre{\v{s}}ar, Gologranc, Milani\v{c}, Rall and Rizzi~\citep{breva}, where motivation came from the domination game as described above.  In addition, Grundy domination presents the worst-case scenario in the process of the online update of a dominating set in the expanding network. 
In $2021$, domination games as well as Grundy domination were comprehensevely surveyed in the book~\citep{game}. 

\medskip
The optimization version of the Grundy domination problem  is to find a GD-sequence (that is, a dominating sequence of maximum length) in a graph $G$. From the computational complexity point of view the decision version of the problem is defined as follows:
\begin{center}
\fbox{
	\parbox{\textwidth}{
		\textsc{Grundy Domination Decision} (GDD)\\
		\textit{Input}: A graph $G = (V,E)$ and $k \in {\mathbb{Z}}^+.$\\
		\textit{Question}: Is there a dominating sequence of $G$ of length at least $k$?}}
\end{center}

\medskip
\noindent In the seminal paper~\citep{breva}, the authors proved that the GDD problem is NP-complete for chordal graphs. They also proved that a GD-sequence in trees, cographs and split graphs can be computed in polynomial time~\citep{breva}. An additional study of Grundy domination in forests was used to give a partial confirmation of a conjecture concerning the Grundy domination number in strong products of graphs~\citep{bell-2021}. Several combinatorial results have also been established for the parameter and its relatives in the literature~\citep{bbgk,camp-2021,er-2019,hay1,l-2019}. Concerning the computational complexity,  it was shown that the GDD problem can be solved in polynomial time for interval graphs and Sierpi\'{n}ski graphs~\citep{gninterval}, as well as on some $X$-join products, lexicographic products and related classes of graphs~\citep{nt-2018}. 
A hierarchy presenting relationships between some classes of graphs that are relevant for this paper are shown in Fig.~\ref{fig:hier}.

%%% figure
\begin{figure}[!h]
\centering
\begin{tikzpicture}[scale=0.9, style=thick]

 \draw (5,5)     node[inner sep=5pt,draw] (a) {General};
 \draw (4,3.5)     node[inner sep=5pt,draw] (b) {Cograph};
 \draw (0,4)     node[inner sep=5pt,draw] (c) {Chordal};
 \draw (6.5,3.5)     node[inner sep=5pt,draw] (d) {Co-bipartite};
 \draw (10,4)     node[inner sep=5pt,draw] (e) {Bipartite};
 \draw (10,2.5)     node[inner sep=5pt,draw] (f) {Chordal Bipartite};
 \draw (10,1)     node[inner sep=5pt,draw] (g) {Convex Bipartite};
 \draw (10,-0.5)     node[inner sep=5pt,draw] (h) {Bipartite Permutation};
 \draw (10,-2)     node[inner sep=5pt,draw] (i) {Chain};
 \draw (-2,2.5)     node[inner sep=5pt,draw] (j) {Interval};
 \draw (0,1)     node[inner sep=5pt,draw] (k) {Split};
 \draw (2,2.5)     node[inner sep=5pt,draw] (l) {Trees};
 
  \draw[->] (a) -- (b);
  \draw[->] (a) -- (c);
  \draw[->] (a) -- (d);
  \draw[->] (a) -- (e);
  \draw[->] (e) -- (f);
  \draw[->] (f) -- (g);
  \draw[->] (g) -- (h);
  \draw[->] (h) -- (i);
   \draw[->] (c) -- (j);
    \draw[->] (c) -- (k);
     \draw[->] (c) -- (l);

\end{tikzpicture}
\caption{Hierarchy of selected graph classes.}
\label{fig:hier}
\end{figure}
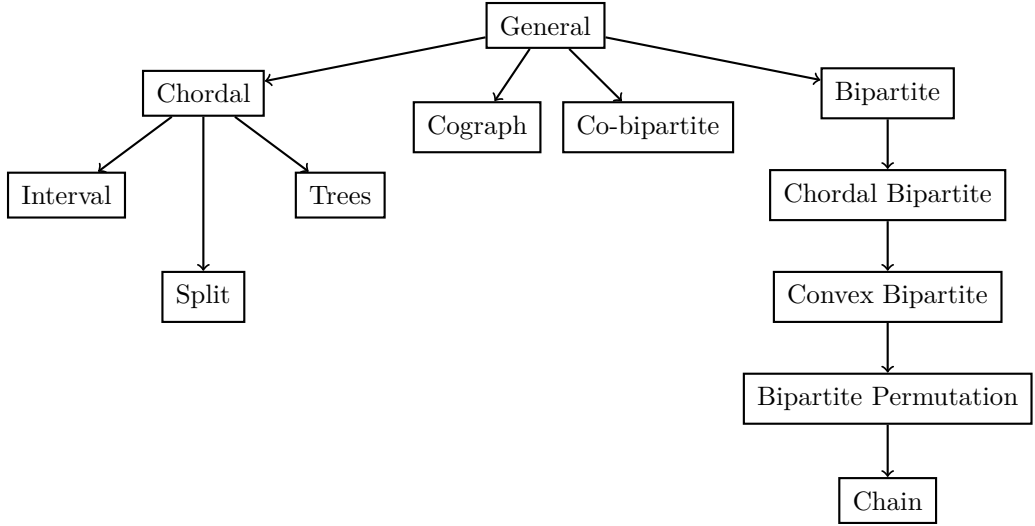

\medskip
A close relation between dominating sequences in graphs and certain covering sequences in hypergraphs was found in the seminal paper~\citep{breva}. Given a hypergraph $\cal H=(X, E)$ with no isolated vertices, an \emph{edge cover} of $\cal H$ is a family of (hyper)edges from ${\cal E}$ whose union is the vertex set $\cal X$. The smallest cardinality of an edge cover of $\cal H$ is the \emph{edge cover number} of $\cal H$ and is denoted by $\rho(\cal H)$.

\medskip
Now, consider a sequence of edges ${\cal C}=(C_1,\ldots,C_r)$ of a hypergraph $\cal H$. If for each $i$, $i \in [r]$,  $C_i$ contains a vertex not contained in $C_j$, for all $j<i$, then ${\cal C}$ is a \emph{legal edge sequence} of $\cal H$. If ${\cal C}=(C_1, C_2, \ldots, C_r)$ is a legal edge sequence and the set $\widehat{\cal C}=\{C_1, C_2, \ldots, C_r\}$ is an edge cover of $\cal H$, then ${\cal C}$ is an \emph{edge covering sequence}.
An edge covering sequence of maximum length in $\cal H$ is a \emph{Grundy covering sequence} of $\cal H$. The length of such a sequence is the \emph{Grundy cover number}, $\rho_{gr}(\cal H)$, of $\cal H$. Given a hypergraph $\cal H=(X, E)$, the \textsc{Grundy Covering} problem is to find a Grundy covering sequence, and thus establishing the Grundy cover number, and its decision version is the \textsc{Grundy Covering Decision} (GCD) problem.
It was shown in \citep{breva} that the GDD problem is NP-complete by reduction from the GCD problem, while for the NP-completeness of the GCD problem a reduction from the classical {\sc Feedback Arc Set} problem was used. 

\medskip
In Section~\ref{sec:pre}, we establish basic notation and mention several preliminary results. In Section~\ref{sec:np}, we prove that the GDD problem is NP-complete even when restricted to bipartite or co-bipartite graphs. In contrast, we present a linear-time algorithm for determining the Grundy domination number of chain graphs in Section~\ref{sec:algo}. In the final section, we add a few comments and open problems. 

\section{Preliminaries}
\label{sec:pre}
Let $[n]=\{1,\ldots, n\}$ for any $n\in\mathbb{N}$. Given a graph $G$, the \emph{neighborhood} of a vertex $x$ is $N_G(x)=\{y\in V(G):\, xy \in E(G)\}$, and the {\em closed neighborhood} of $x$ is $N_G[x]=N_G(x)\cup\{x\}$. Vertices $u$ and $v$ in a graph $G$ are {\em closed twins} (respectively, {\em open twins}) if $N_G[u]=N_G[v]$ (respectively, $N_G(u)=N_G(v)$).  We may omit the indices in the above definitions if the graph $G$ is understood from the context. %Note that in any (Grundy) dominating sequence of $G$ vertex $u$ appears before vertex $v$ if the closed neighborhood of $v$ contains the closed neighborhood of $u$. We state this property as a proposition below.
The following statement has a straightforward proof using definitions. 

\medskip
\begin{proposition}\label{pro:basic}
Let $S$ be a (Grundy) dominating sequence in a graph $G$. If $u,v \in V(G)$ such that $N[u] \subseteq N[v]$ and $u,v \in \hat{S}$, then $u$ appears before $v$ in $S$.
\end{proposition}

\medskip
If $S$ is a closed neighborhood sequence, then we say that $v_i$ \emph{footprints} the vertices from $N[v_i] \setminus \cup_{j=1}^{i-1}N[v_j]$, and that $v_i$ is the \emph{footprinter} of every vertex $u\in N[v_i] \setminus \cup_{j=1}^{i-1}N[v_j]$.
If $S_1=(v_1, v_2, \ldots , v_n)$ and $S_2=(u_1, u_2, \ldots , u_m)$, $n,m \geq \textcolor{black}{1},$ are two sequences, then the {\em concatenation} of $S_1$ and $S_2$ is the sequence $S_1 \oplus S_2=(v_1, v_2, \ldots , v_n, u_1, u_2, \ldots , u_m)$.

\medskip
Given a hypergraph $\cal H =(X,E)$, a \emph{legal transversal sequence} is a sequence $S=(v_1,  \ldots, v_k)$ of vertices from $\cal X$ such that for each $i$ there exists an edge $\mathcal{E}_i \in \mathcal{E}$ such that $v_i \in \mathcal{E}_i$ and $v_j \notin \mathcal{E}_i$ for all $j$, where $j<i$.  The maximum length of a legal transversal sequence in a hypergraph $\cal H$ is denoted by $\taugr({\cal H})$. The following result was proved in~\citep[Proposition 8.3]{total}.

\medskip
\begin{proposition}
\label{p:rho-tau}
For any hypergraph ${\cal H}$ we have $\taugr({\cal H})=\rhogr({\cal H})$.
\end{proposition}

\medskip
A set of vertices $A \subseteq V(G)$ is called an \emph{independent set} of $G$ if no two vertices of $A$ are adjacent in $G$. A \emph{maximum independent set} is an independent set of maximum cardinality. The size of a maximum independent set in $G$ is the \emph{independence number} of $G$ and is denoted by $\alpha(G)$.

\smallskip
A \emph{bipartite graph} is a graph whose vertex set can be partitioned into two independent sets. A \emph{co-bipartite graph} is a graph, which is the complement of a bipartite graph. A bipartite graph $G = (X, Y,E)$ is a \textit{chain graph} if the neighborhoods of the vertices of $X$ form a chain, that is, the vertices of $X$ can be linearly ordered, say  $\{x_1,x_2, \ldots ,x_{n_1}\}$ such that $N(x_1) \subseteq N(x_2) \subseteq \cdots \subseteq N(x_{n_1})$, where $n_1 = \vert X\vert$. If $G=(X,Y,E)$ is a chain graph, then one can easily see that the neighborhoods of the vertices of $Y$ also form a chain. If $n_2 = \vert Y\vert$, an ordering $\alpha=(x_1,x_2, \ldots ,x_{n_1},y_1,y_2, \ldots ,y_{n_2})$ is called a \textit{chain ordering} if  $N(x_1) \subseteq N(x_2) \subseteq \cdots \subseteq N(x_{n_1})$ and $N(y_1) \supseteq N(y_2) \supseteq \cdots \supseteq N(y_{n_2})$. 

\smallskip
If a vertex $x_{i}$ appears before $x_{j}$ in chain ordering, we write $x_{i}<x_{j}$. Given a chain graph $G$, a chain ordering of $G$ can be computed in linear time \citep{hegge}. Now, suppose $A = \{x_{i_1}, x_{i_2}, \ldots, x_{i_r}\}, B =\{x_{j_1}, x_{j_2}, \ldots, x_{j_{r'}}\}$ be two disjoint subsets of $X$ such that $x_{i_p} < x_{i_{q}}$ and $x_{j_p} < x_{j_{q}}$ for $p<q$. Then $(A) \oplus (B)$ denotes the sequence $(x_{i_1}, \ldots, x_{i_r}, x_{j_1},  \ldots, x_{j_{r'}})$. Similar notation can be used for vertices in $Y$. 

\smallskip
Recall that two vertices $u,v$ are called open twins if $N(u) = N(v)$. We define a relation $R$ on $X$ such that vertices $x,x' \in X$ are related by $R$ if and only if they are open twins. Clearly, $R$ is an equivalence relation, and let $X_1, X_2, \ldots, X_k$ be the parts in $X$ that arise from $R$. Without loss of generality, for a chain graph $G=(X,Y,E)$, we may also assume that $N(X_{1})\subseteq N(X_{2}) \cdots  \textcolor{black}{\subseteq} N(X_{k})$. We denote \textcolor{black}{the} set of vertices in $N(X_1)$ by $Y_1$. For $i \in \{2, 3, \ldots, k\}$, let $Y_i = N(X_i) \setminus \cup_{j=1}^{i-1} N(X_j)$. %Note that if $u,v$ are two vertices of $X_i$, where $i \in [k]$, then $u,v$ happen to be open twins. 
It is easy to see that every two vertices from $Y_j$ are open twins for all $j \in [k]$. The following statement follows from the construction of the sets $X_i$ and $Y_j$.

\medskip
\begin{Ob}\label{chainpartition}
Let $X_1, X_2, \ldots, X_k$ and $Y_1, Y_2, \ldots, Y_k$ be the subsets of vertices of $G$ as defined above. Then, $N(X_i) = \bigcup_{r=1}^{i} Y_r$ and $N(Y_j) = \bigcup_{r=j}^{k}X_r$ for each $i,j \in [k]$.
\end{Ob}

\section{NP-completeness results}
\label{sec:np}
As mentioned in the introduction, the GDD problem is NP-complete for general graphs, and also when restricted to chordal graphs. In the following subsections, we prove that the problem remains NP-complete for bipartite and co-bipartite graphs.

\subsection{Bipartite graphs}

In this subsection, we prove the NP-completeness of the GDD problem for bipartite graphs. We reduce the GCD problem for hypergraphs to the GDD problem for bipartite graphs. Given a hypergraph $\cal H = (X,E)$ with $| \mathcal X| = n$ and $ |\mathcal{E}| = m$, $(n,m \geq 2)$, we construct an instance $G_{\cal H}$ of the GDD problem, where $G_{\cal H}$ is a bipartite graph, as follows.

\smallskip
Let $V(G_{\cal H}) = A \cup X' \cup {\mathcal{E}}' \cup B$, where $A =\{a_1, a_2, \ldots, a_n\}$ and $B =\{b_1, b_2, \ldots, b_m\}$. The sets $X'$ and $ {\mathcal{E}}'$ contain $n$ and $m$ vertices respectively, where each vertex of $X'$ corresponds to a vertex of $\cal X$ in the hypergraph $\cal H$ and each vertex of ${\mathcal{E}}'$ correspond to an edge of $\cal H$. For an edge ${\mathcal{E}}_i \in \mathcal{E}$, we denote the corresponding vertex in ${\mathcal{E}}'$ by $e_i$. All sets $A,X',{\mathcal{E}}',$ and $B$ are indepedent sets in $G_{\cal H}$, and $A\cup X'$ induces the complete bipartite graph $K_{n,n}$, while ${\mathcal{E}}'\cup B$ induces the complete bipartite graph $K_{m,m}$.  Now, a vertex $x$ of $X'$ is adjacent to a vertex of $e_i \in {\mathcal{E}}'$ in $G_{\cal H}$ if and only if $x \in {\mathcal{E}}_i$ in $\cal H$.

\smallskip
 Clearly, $G_{\cal H}$ is a bipartite graph. See Fig.~\ref{fig:2}, which presents  the construction of the graph $G_{\cal H}$ from a hypergraph $\cal H$, which is given by $(\mathcal X =\{x_1, x_2, x_3, x_4\}, \mathcal{E} = \{\mathcal{E}_1, \mathcal{E}_2, \mathcal{E}_3, \mathcal{E}_4, \mathcal{E}_5\})$, where $\mathcal{E}_1 = \{x_1, x_2, x_4\}$, $\mathcal{E}_2 = \{x_2, x_3\}$, $\mathcal{E}_3 = \{x_1, x_2\}$, $\mathcal{E}_4 = \{x_2, x_3, x_4\}$ and $\mathcal{E}_5 = \{x_1,x_3, x_4\}$.

\begin{figure}[!h]
\begin{center}
\begin{tikzpicture}[scale=0.7]
    % Place nodes

	\vertex (a1) at (-5, 0) [fill=black]{};
	\vertex (a2) at (-5, 2) [fill=black]{};
	\vertex (a3) at (-5, 4) [fill=black]{};
	\vertex (a4) at (-5, 6) [fill=black]{};
	
	\vertex (x1) at (-2, 0) [fill=black,label=below:$x_1$]{};
	\vertex (x2) at (-2, 2) [fill=black,label=below:$x_2$]{};
	\vertex (x3) at (-2, 4) [fill=black,label=above:$x_3$]{};
	\vertex (x4) at (-2, 6) [fill=black,label=above:$x_4$]{};
	
	\vertex (e1) at (2, 0) [fill=black,label=below:$e_1$]{};
	\vertex (e2) at (2, 2) [fill=black,label=below:$e_2$]{};
	\vertex (e3) at (2, 4) [fill=black,label=above:$e_3$]{};
	\vertex (e4) at (2, 6) [fill=black,label=above:$e_4$]{};
	\vertex (e5) at (2, 8) [fill=black,label=above:$e_5$]{};
	
	\vertex (b1) at (5, 0) [fill=black]{};
	\vertex (b2) at (5, 2) [fill=black]{};
	\vertex (b3) at (5, 4) [fill=black]{};
	\vertex (b4) at (5, 6) [fill=black]{};
	\vertex (b5) at (5, 8) [fill=black]{};
	
	\draw (-5,-1.2) node {{\large $A$}};
	\draw (-2,-1.2) node {{\large $X'$}};
	\draw (2,-1.2) node {{\large ${\cal E}'$}};
	\draw (5,-1.2) node {{\large $B$}};
	
	\draw (a1)--(x1)--(a2)--(x2)--(a3)--(x3)--(a4)--(x4)--(a1);	
	\draw (a1)--(x2)--(a3)--(x4)--(a2)--(x3)--(a1);	
	\draw (x1)--(a3);	
	\draw (x1)--(a4)--(x2);
	
	\draw (b1)--(e1)--(b2)--(e2)--(b3)--(e3)--(b4)--(e4)--(b1)--(e5)--(b2);	
	\draw (b1)--(e2)--(b3)--(e4)--(b2)--(e3)--(b1);	
	\draw (b4)--(e1)--(b3)--(e5)--(b4);	
	\draw (b3)--(e4)--(b2);
	\draw (e5)--(b5)--(e4);
	\draw (e3)--(b5)--(e2); \draw (e1)--(b5);
	
	\draw (x1)--(e1)--(x2); \draw (e1)--(x4);
	\draw (x2)--(e2)--(x3); 
	\draw (x1)--(e3)--(x2); 
	\draw (x2)--(e4)--(x3); \draw (e4)--(x4);
	\draw (x1)--(e5)--(x3); \draw (e5)--(x4);

\end{tikzpicture}
\end{center}
\caption{Illustration of the construction of $G_{\cal H}$ from $\cal H$.}
\label{fig:2}
\end{figure}
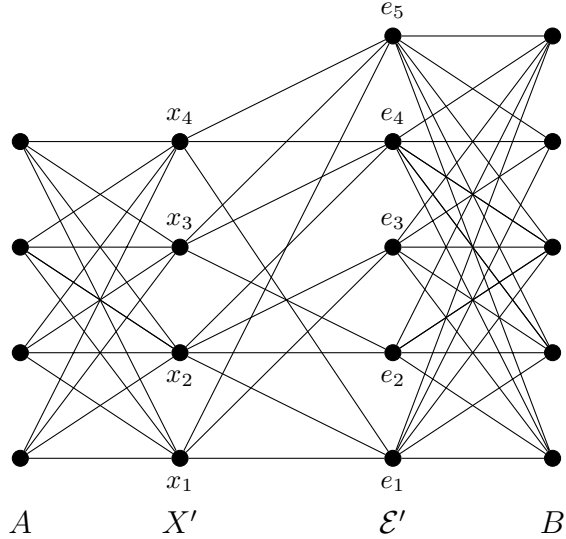

\begin{theorem}\label{nphard_bip_equiv}
If $\cal H  = (X,E)$ is a hypergraph with $|\mathcal X| = n$ and $ |\mathcal{E}| = m$, where $n,m \geq 2$, then ${\rho}_{gr} (\mathcal H) \geq k$ if and only if ${\gamma}_{gr} (G_{\cal H}) \geq n+m+k$.
\end{theorem}

\begin{proof}
In the proof, let $G=G_{\cal H}$. First, let ($\mathcal{E}_1, \mathcal{E}_2, \ldots, \mathcal{E}_{k'})$ be an edge covering sequence of length at least $k$ in $\cal H$. Then the sequence $(b_1, b_2, \ldots, b_m, e_1, e_2, \ldots, e_{k'}, a_1, a_2, \ldots, a_n)$ is a dominating sequence of length at least $n+m+k$ in $G$. Hence, ${\gamma}_{gr} (G) \geq n+m+k$.

\smallskip
For the converse, let us assume that ${\gamma}_{gr} (G) \geq n+m+k$ for some positive integer $k$. If $X' \cap \hat{S} \neq \emptyset$ for some dominating sequence $S$ of $G$, then $x_0$ denotes the first vertex in $S$ coming from $X'$, and if ${\mathcal{E}}' \cap \hat{S} \neq \emptyset$ for some dominating sequence $S$ of $G$, then $e_0$ denotes the first vertex in $S$ coming from ${\mathcal{E}}'$. First, we prove two auxiliary claims.

\begin{cl}\label{GDN_claim1}
There exists a dominating sequence $S$ of length at least $n+m+k$ in $G$ such that all vertices of $A$ appear in $S$ and if $X' \cap \hat{S} \neq \emptyset$, then all vertices of $A$ appear before $x_0$.
\end{cl}
\begin{proof}
Let $S$ be a dominating sequence of length at least $n+m+k$ in $G$ (which exists, since ${\gamma}_{gr} (G) \geq n+m+k$). Suppose there exists a vertex $a_i \in A$, which is not appearing in $S$. Then, there exists a vertex from $X'$ which is appearing in $S$ to footprint $a_i$. Hence, $X' \cap \hat{S} \neq \emptyset$ and $x_0$ footprints $a_i$. Let $P$ denotes the set of  vertices appearing before $x_0$ and $Q$ be the set of vertices appearing after $x_0$ in $S$. Now, two cases are possible.

\smallskip
\noindent
\textbf{Case 1:} {$Q \cap A = \emptyset$}.\\
In this case, either $P\cap A=\emptyset$ or $P$ contains some vertices of $A$. So, first assume that $P$ contains no vertex of $A$. Then, we see that no vertex of $A$ appears in the sequence $S$. If $x_0$ does not footprint any vertex in ${\mathcal{E}}'$, then we modify $S$ by appending vertices of $A$ in the order $(a_1, a_2, \ldots, a_n)$ just before $x_0$ and removing all vertices from $Q$ that footprinted a vertex of $X'$ along with the vertex $x_0$. Otherwise, if $x_0$ footprints some vertices of ${\mathcal{E}}'$, we perform the same modification without removing $x_0$. \textcolor{black}{Note that the number of vertices of $Q$ that footprint a vertex of $X'$ are at most $n -1$}. In either case, we removed at most $n$ vertices and we added $n$ new vertices to $S$, by which the so modified sequence $S$ is a dominating sequence of length at least $n+m+k$ in $G$, which satisfies the statement of the claim. 

Now, if $P$ contains some vertices of $A$, then no vertex of $Q$ footprints any vertex of $X'$. Again, if $x_0$ footprints only vertices of $A$, then we modify $S$ by appending vertices of {$A \setminus P$} in any order just before $x_0$ and removing the vertex $x_0$. Otherwise, if $x_0$ footprints also some vertices of ${\mathcal{E}}'$, we perform the same modification, but without removing $x_0$. In either case, we removed at most $1$ vertex and we added at least $1$ new vertex to the sequence $S$. With this, the so modified sequence $S$ is a dominating sequence of length at least $n+m+k$ in $G$, which satisfies the statement of the claim. 

\smallskip
\noindent
\textbf{Case 2:} { $|Q \cap A| = 1$}.\\
In this case, $P$ contains no vertex of $A$. Let $a_j$ be the vertex from $Q \cap A$ appearing in $S$. Note that the vertices in $S$, which footprint only vertices from $X'$, do not appear after $a_j$ in $S$. Note that there are at most $n-2$ vertices that appear in $S$ between $x_0$ and $a_j$ and footprint a vertex of $X'$, and denote the set of these vertices by $Q'$. If $x_0$ does not footprint any vertex in ${\mathcal{E}}'$, then we modify $S$ by appending vertices of $A$ in the order $(a_1, a_2, \ldots, a_n)$ just before $x_0$ and then removing all vertices of $Q'\cup\{x_0,a_j\}$. Otherwise, if $x_0$ footprints also some vertices of ${\mathcal{E}}'$, we perform the same modification without removing $x_0$. In either case, we removed at most $n-1$ vertices and we added $n-1$ new vertices to $S$. Hence, the so modified sequence $S$ is a dominating sequence of length at least $n+m+k$ in $G$, which satisfies the statement of the claim. 
\end{proof}

\smallskip
By the above claim, there exists a dominating sequence $S$ of length $n+m+k$ such that all vertices of $A$ appear in $S$ and they appear before $x_0$ if $X' \cap \hat{S} \neq \emptyset$.
The proof of Claim~\ref{GDN_claim2} follows similar lines as the proof of the Claim~\ref{GDN_claim1}.

\begin{cl}\label{GDN_claim2}
There exists a dominating sequence $S$ of length at least $n+m+k$ in $G$ such that all vertices of $B$ appear in $S$ and if ${\mathcal{E}}' \cap \hat{S} \neq \emptyset$, then all vertices of $B$ appear before $e_0$.
\end{cl}

Combining the above two claims we infer that there exists a dominating sequence $S$ of length at least $n+m+k$ in $G$ such that $|(X' \cup {\mathcal{E}}') \cap \hat{S}| \geq k$.
\begin{cl}\label{GDN_claim3}
Either $X' \cap \hat{S} = \emptyset$ or ${\mathcal{E}}' \cap \hat{S} = \emptyset$.
\end{cl}

\begin{proof}
If $X' \cap \hat{S} = \emptyset$, we are done. So, assume that $X' \cap \hat{S} \neq \emptyset$ and ${\mathcal{E}}' \cap \hat{S} \neq \emptyset$. Now, either $e_0$ appears before $x_0$ or $e_0$ appears after $x_0$. In the former case, we see that before the vertex $x_0$, all vertices of $G$ are footprinted (and thus dominated) using Claims~\ref{GDN_claim1} and~\ref{GDN_claim2}. So, $x_0$ does not footprint any vertex implying that this case is not possible. Similarly, the latter case is also not possible, which proves the claim.
\end{proof}

\smallskip
Now, if $X' \cap \hat{S} = \emptyset$, then we have that $|{\mathcal{E}}' \cap \hat{S}| \geq k$. In addition, by Claim~\ref{GDN_claim2}, since all vertices of $B$ appear in $S$ before any vertex of ${\mathcal{E}}'$ appears in $S$, the subsequence of $S$ of vertices in ${\mathcal{E}'}$ corresponds to an edge covering sequence in the hypergraph $\cal H$, which is of length at least $k$. 
Thus, ${\rho}_{gr} (\mathcal H) \geq k$, as desired.

\smallskip
Otherwise, if ${\mathcal{E}}' \cap \hat{S} = \emptyset$, then we derive that $|X' \cap \hat{S}|\ge k$, where the subsequence formed by vertices of $X'\cap \widehat{S}$ corresponds to a legal  transversal sequence of the hypergraph $\cal H$ of length at least $k$. By Proposition~\ref{p:rho-tau}, $\taugr({\cal H})=\rhogr({\cal H})$, and so ${\rho}_{gr} (\mathcal H) \geq k$. The proof of the theorem is complete. 
\end{proof}

Based on Theorem~\ref{nphard_bip_equiv} and earlier discussions we immediately derive the main result of this section.

\medskip
\begin{theorem}
The GDD problem is NP-complete for bipartite graphs.
\end{theorem}

\subsection{Co-bipartite graphs}
In this subsection, we prove the NP-completeness of the GDD problem for co-bipartite graphs. Here, we reduce the GDD problem for general graphs to the GDD problem for co-bipartite graphs. Given a graph $G = (V,E)$, where $V=\{v_1, v_2, \ldots, v_n\}$, we construct an instance $G'=(V_1 \cup V_2, E')$ of the GDD problem, where $G'$ is a co-bipartite graph, as follows.

\smallskip
The vertex set of $G'$ is  $V_1 \cup V_2$, where $V_1=\{v_i^1: v_i \in V\}$ and $V_2=\{v_i^2: v_i \in V\}$. The set of edges of $G'$ is given by $\{v_i^1 v_j^1 : 1\leq i<j \leq n\} \cup \{v_i^2 v_j^2 : 1\leq i<j \leq n\} \cup \{v_i^1 v_j^2 : v_j \in N_G[v_i], i,j \in [n] \}$. Note that $G'$ is a co-bipartite graph. Fig.~\ref{cobip} provides an illustration of the construction of $G'$ from $G$. %Next, we prove the following claim.

\begin{figure}[!h]
\begin{center}
\begin{tikzpicture}[scale=0.7]
    % Place nodes
    
    \vertex(a) at (-5, 2) [fill=black,label=left:$a$]{};
	\vertex (b) at (-5, 0) [fill=black,label=left:$b$]{};
	\vertex (c) at (-5, -2) [fill=black,label=left:$c$]{};
	
	 \vertex(a1) at (5, 2) [fill=black,label=left:$a^1$]{};
	\vertex (b1) at (6, 0) [fill=black,label=left:$b^1$]{};
	\vertex (c1) at (5, -2) [fill=black,label=left:$c^1$]{};
	 \vertex(a2) at (10, 2) [fill=black,label=right:$a^2$]{};
	\vertex (b2) at (9, 0) [fill=black,label=right:$b^2$]{};
	\vertex (c2) at (10, -2) [fill=black,label=right:$c^2$]{};

	\draw (-6.5,0) node {{\Large $G:$}};
	\draw [-stealth,line width=5pt](-2,0) -- (0,0);
	
	\draw (3,0) node {{\Large $G':$}};
	
	\draw (a)--(b)--(c);
    \draw (a1)--(b1)--(c1)--(a1);
    \draw (a2)--(b2)--(c2)--(a2);
    \draw (a1)--(b2)--(c1)--(c2);
    \draw (a1)--(a2)--(b1)--(c2);
    \draw (b1)--(b2);
    
    \draw[rounded corners=10pt,dashed] (4,-2.7) rectangle (6.5,2.7);      
    \draw[rounded corners=10pt,dashed] (8.5,-2.7) rectangle (11,2.7);
    \draw (4,-3) node {{\large $V_1$}};
    \draw (11,-3) node {{\large $V_2$}};

\end{tikzpicture}
\end{center}
\caption{Construction of a co-bipartite graph $G'$ from a graph $G$, where $G=P_3$.}
\label{cobip}
\end{figure}
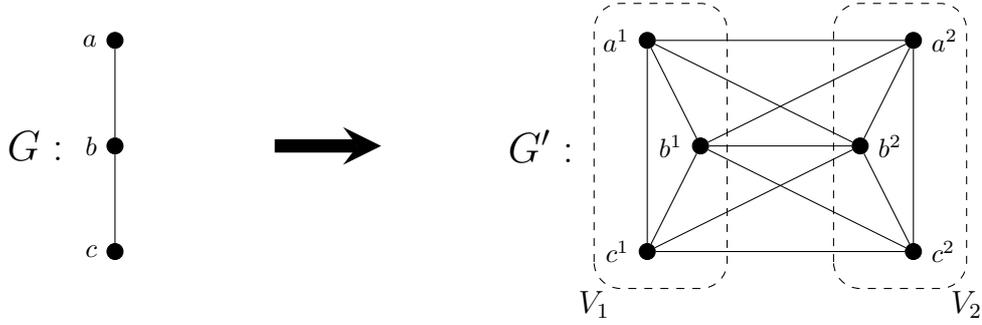

\medskip
\begin{cl}\label{cl:cobip} For a positive integer $k$, 
$\gamma_{gr}(G) \geq k$ if and only if $\gamma_{gr} (G') \geq k$.
\end{cl}
\begin{proof}
First, let $S = (u_1, u_2, \ldots, u_t)$ be a dominating sequence of $G$ of length $t$, where $t \geq k$. Then $S' = (u_1^1, u_2^1, \ldots, u_t^1)$ is a dominating sequence of length at least $k$ in $G'$. Indeed, if $v_i$ is a vertex footprinted by $u_i$ with respect to $S$, then $v_i^2$ is footprinted by $u_i^1$ with respect to $S'$. 

\smallskip
Conversely, let $S = (w_1, w_2, \ldots, w_t)$ be a dominating sequence of length $t$ in $G'$, where $t \geq k$ and $w_i \in V_1 \cup V_2$ for all $i \in [t]$. Without loss of generality, we may assume that $w_1  \in V_1$. Note that there can be at most one vertex from $V_2$ in $S$. If there is no such vertex, then the sequence $S$ corresponds to a sequence of vertices in $G$ of length at least $k$, which is a dominating sequence of $G$. Now, suppose there exists a vertex from $V_2$ in $S$. Clearly, it has to be the last vertex of $S$, and let $w_t$ be vertex $v_i^2 \in V_2$. Note that $v_i^2$ appears in $S$ to footprint a vertex $v_j^2 \in V_2$ and $v_j^1 \notin \hat{S}$. We modify $S$ by replacing the vertex $v_i^2$ with the vertex $v_j^1$ and get a new sequence $S'$, which is again a dominating sequence of $G'$ of length at least $k$. Since $S'$ contains only vertices from $V_1$, it corresponds to a sequence of vertices in $G$ of length at least $k$, which is a dominating sequence of $G$. This completes the proof of the converse direction of the statement. 
\end{proof}

Now, we are ready to state the announced result.

\medskip
\begin{theorem}
The GDD problem is NP-complete for co-bipartite graphs.
\end{theorem}

%%%%%%%%%%%%%%%%%%%%%%%%%%%
\section{Efficient algorithm for chain graphs}
\label{sec:algo}
%%%%%%%%%%%%%%%%%%%%%%%%%%%%%

In this section, we give a linear-time algorithm to compute a GD-sequence of a chain graph. Before discussing the main idea for chain graphs, we first give the Grundy domination number of a complete bipartite graph which is a subclass of chain graphs. The proof of this result is \textcolor{black}{straightforward}.

\medskip
\begin{proposition}\label{ob:GD_complete}
If $G = (X, Y,E)$ is a complete bipartite graph, then $\gamma_{gr} (G) = max\{|X|, |Y|\}$.
\end{proposition}

\medskip
{Recall that a chain graph $G=(X,Y,E)$ has a chain ordering $\alpha$, and based on the equivalence relation joining open twins, the sets $X$ and $Y$ partition into subsets $X_1, X_2, \ldots, X_k$ and $Y_1, Y_2, \ldots, Y_k$, respectively. (Also recall that $N(X_i) = \bigcup_{r=1}^{i} Y_r$ and $N(Y_j) = \bigcup_{r=j}^{k}X_r$ for each $i,j \in [k]$.) Since the case $k=1$ yields a complete bipartite graph, in the rest of this section we only consider the chain graphs with $k\ge 2$. We also assume that a chain graph $G=(X,Y,E)$ is given along with the chain ordering $\alpha$ and the partitions of $X$ and $Y$ into $k$ subsets. 

\smallskip
The proof of the following observation is again easy, and hence is omitted.

\smallskip
\begin{Ob}\label{basicproperty1}
\textcolor{black}{Let $A \subseteq V(G)$ be a set of open twins in an arbitrary graph $G$. Then there exists a GD-sequence $S$ of $G$ such that all vertices of $A\cap \hat{S}$ appear together in $S$.}
\end{Ob}

\smallskip 
\begin{Ob}\label{basicproperty2}
\textcolor{black}{Let $S$ be a GD-sequence of an arbitrary graph $G$ and $A \subseteq V(G)$ be a set of open twins in $G$ such that $A \cap \hat{S} \neq \emptyset$. If the first vertex of $A$ in $S$ footprints itself, then there exists a GD-sequence of $G$ in which all vertices of $A$ appear and they appear together in that sequence.}
\end{Ob}
\begin{proof}
\textcolor{black}{If $A \cap \hat{S} = A$, we have nothing to prove. So, suppose that $A \cap \hat{S} \neq A$. Then, there exists a vertex $a \in A$, which is not in $S$. Thus there exists a vertex $b \in \hat{S}$, which footprints $a$. Note that $b$ appears after all vertices of $A \cap \hat{S}$. Now, we modify $S$ by replacing $b$ with the vertex $a$. By doing this repeatedly, we get a new GD-sequence $S'$ of $G$ which contains all vertices of $A$. We can rearrange all vertices of $A$ so that all vertices of $A$ appear together in $S'$.}
\end{proof}

In the remainder of the section we assume that $G$ is a chain graph with the partition of its vertex set as described earlier.
\smallskip
\begin{Ob}\label{basicproperty3}
\textcolor{black}{There exists a GD-sequence $S$ of $G$ such that for every $i \in [k]$, we have $X_i \cap \hat{S} \neq \emptyset$ or $Y_i \cap \hat{S} \neq \emptyset$}.
\end{Ob}
\begin{proof}
\textcolor{black}{Let $S$ be a GD-sequence such that} $X_i \cap \hat{S} = \emptyset$ and $Y_i \cap \hat{S} = \emptyset$. This implies that there exist vertices $x \in \cup_{r=i+1}^{k}X_r$ and $y \in \cup_{r=1}^{i-1}Y_r$ in $S$ to footprint the vertices of $Y_i$ and $X_i$ respectively. Note that $i \notin \{1,k\}$. Now, if $x$ appears before $y$ in $S$, then modifying the sequence $S$ by replacing $y$ by all vertices of $X_i$ gives another GD-sequence of $G$, which includes some vertices from $X_i$. (Note that if $y \in Y_j$, then $X_j \cap \hat{S} \neq \emptyset$. To see this, assuming that $X_j \cap \hat{S} = \emptyset$ implies that $y$ footprints vertices of both $X_j$ and $X_i$ in $S$. In this case, replacing the vertex $y$ with all vertices of $X_i$ and $X_j$ results in a new dominating sequence of $G$ of length bigger than $S$, a contradiction.) Similarly, if $y$ appears before $x$ in $S$, then modifying the sequence $S$ by replacing $x$ by all vertices of $Y_i$ gives another GD-sequence of $G$, which includes some vertices from $Y_i$. Hence, there exists a GD-sequence $S$ of $G$ such that $X_i \cap \hat{S} \neq \emptyset$ or $Y_i \cap \hat{S} \neq \emptyset$.
\end{proof}

\smallskip
\textcolor{black}{Let $A$ be a set of open twins in $G$. If $|A \cap \hat{S}| \geq 2$, then we see that the first vertex of $A$ in $S$ footprints itself. Thus we can assume that, $A \subseteq \hat{S}$, due to Observation~\ref{basicproperty2}. Hence, we have, $|A \cap \hat{S}| \leq 1$ or $A \subseteq \hat{S}$ for any set of open twins $A$ in $G$. Note that the each of the sets $X_1, X_2, \ldots, X_k, Y_1, Y_2, \ldots, Y_k$ is a set of open twins in $G$.}

\smallskip
\textcolor{black}{Now, based on the Observations~\ref{basicproperty1}, \ref{basicproperty2} and~\ref{basicproperty3}, whenever we consider a GD-sequence $S$ of $G$, we assume that $S$ satisfies the following, for the rest of this section:\\
(1) For each $i \in [k], |X_i \cap \hat{S}| \leq 1$ or $X_i \subseteq \hat{S}$. If $X_i \subseteq \hat{S}$, then all vertices of $X_i$ appear together in $S$.\\
(2) For each $i \in [k], |Y_i \cap \hat{S}| \leq 1$ or $Y_i \subseteq \hat{S}$. If $Y_i \subseteq \hat{S}$, then all vertices of $Y_i$ appear together in $S$.\\
(3) For each $i \in [k],~ X_i \cap \hat{S} \neq \emptyset$ or $Y_i \cap \hat{S} \neq \emptyset$.}

\vspace{2 mm}

{Now, let $S$ be a GD-sequence of $G$. Then $S$ is one of the following type}:
\begin{center}$(a)$ $X \cap \hat{S} = \emptyset$, \hskip 1cm $(b)$ $Y \cap \hat{S} = \emptyset$, \hskip 1cm $(c)$ $X \cap \hat{S} \neq \emptyset$ and $Y \cap \hat{S} \neq \emptyset$. 
\end{center}
We call the corresponding GD-sequences $S$ to be of \emph{type~(a)}, {\em type~(b)}, and \emph{type~(c)}, respectively. 

\smallskip
\begin{lemma}\label{lemma:chain1}
Let $S^* = (v_1, v_2, \ldots, v_p)$ be a GD-sequence of $G$ of type (c). The following statements hold:
\begin{enumerate}[(1)]
\item If $v_1 \in Y_1$, then there exists a type (a) GD-sequence of $G$.
\item If $v_1 \notin Y_1$, then there exists a GD-sequence $S$ of $G$ such that $\cup_{r=1}^{i} X_r \subseteq \hat{S}$ for some $i \in [k]$.
\end{enumerate}
\end{lemma}

\begin{proof}
First, we assume that $v_1 \in Y_1$. In this case, all vertices of $X$ are footprinted by $v_1$. So, all the vertices $v_2, \ldots, v_p$ appear to footprint vertices of $Y \setminus \{v_1\}$ only. This implies that $\gamma_{gr}(G) \leq |Y|$. So, the sequence $S = (y_{n_2}, y_{n_2 - 1}, \ldots, y_{1})$ is also a GD-sequence of $G$ and it is of type (a).
Next, we assume that $v_1 \notin Y_1$. Since $S^*$ contains vertices from both $X$ and $Y$, we have two cases to consider.

\smallskip
\noindent
\textbf{Case 1:} \underline{$ v_1 \in X$}.\\
Let $v_1 \in X_i ~(i \in [k])$. So, we get that $v_1$ footprints all vertices of $N(X_i) \cup \{v_1\}$. Now, let $u$ be a vertex of $X$ such that $N(u) \subseteq N(v_1)$. If $u$ is footprinted by some vertex from $N(u)$, we modify the sequence $S^*$ as follows. We remove the footprinter of $u$ from $S^*$ and include $u$ just after $v_1$ and get a new sequence. But, if $u$ is footprinted by itself, then we relocate $u$ in $S$ by putting it just after $v_1$. We repeat the respective modifications for each vertex $u$ such that $N(u) \subseteq N(v_1)$ and get a new sequence $S$ which remains a GD-sequence of $G$. We again modify the ordering of vertices in $S$, so that it satisfies all the properties of Observations~\ref{basicproperty1}, \ref{basicproperty2} and~\ref{basicproperty3}, if required. Hence, we see that $\cup_{r=1}^{i} X_r \subseteq \hat{S}$ for some $i \in [k]$.

\medskip
\noindent
\textbf{Case 2:} \underline{$ v_1 \in Y$}.\\
Let $v_1 \in Y_i ~(i \in [k])$. Note that $i > 1$. Here, we may assume that all vertices $u$ from $Y$, where $N(u) \supseteq N(v_1)$, are in $S^*$ and all vertices of $\cup_{r=i}^{k} Y_r \setminus \{v_1\}$ appear together just after $v_1$. This is ensured because we can do modifications similar to the case 1, if it is not true. We rename vertices of $S^*$ again by $(v_1, v_2, \ldots, v_p)$, if necessary. Now, let $v_t$ be the vertex with the smallest index in the ordering $(v_2, v_3, \ldots, v_p)$ such that $v_t \in X$. Suppose that $v_t \in X_{t'}$. Note that there exists an integer $r' \in \{2, 3, \ldots, k\}$ such that all vertices of $\cup_{r=r'}^{k} (X_r \cup Y_r)$ are footprinted before the vertex $v_t$ {appears in $S$}, and all remaining vertices of $G$ are not footprinted.
If $v_t \in \cup_{r=1}^{r'-1} X_r$, we get that $v_t$ footprints all vertices of $N(v_t) \cup \{v_t\}$. In this case, we put all vertices $u$ of $X$ such that $N(u) \subseteq N(v_t)$ just after $v_t$ in any order, remove vertices from $N(u) \cap \widehat{S^*}$ which were appearing to footprint the vertex $u$ and rearrange all vertices so that the new sequence $S$ satisfies all the properties of Observation~\ref{basicproperty1}. Thus, we get that $\cup_{r=1}^{t'} X_r \subseteq \hat{S}$.
But, if $v_t \in \cup_{r=r'}^{k} X_r$, then $v_t$ footprints
all vertices of $\cup_{r=1}^{r'-1} Y_r$. So, we may assume that $v_t \in X_{r'}$. Again, we put all vertices $u \in \cup_{r=1}^{r'-1}X_r$ just after $v_t$ in any order, remove vertices from $N(u) \cap \widehat{S^*}$ which were appearing to footprint the vertex $u$ and rearrange all vertices so that the new sequence $S$ satisfies all the properties of Observation~\ref{basicproperty1}. Thus, we get that $\cup_{r=1}^{r'-1} X_r \subseteq \hat{S}$.

Therefore, there exists a GD-sequence $S$ of $G$ such that $\cup_{r=1}^{i} X_r \subseteq \hat{S}$ for some $i \in [k]$.
\end{proof}

\smallskip
Analogous to Lemma~\ref{lemma:chain1}, we give a symmetric lemma for the set $Y$ of $G$, whose proof follows similar lines and is omitted. 

\medskip
\begin{lemma}\label{lemma:chain2}
Let $S^* = (v_1, v_2, \ldots, v_p)$ be a GD-sequence of $G$ of type (c). The following statements hold:
\begin{enumerate}[(1)]
\item If $v_1 \in X_k$, then there exists a type (b) GD-sequence of $G$.
\item If $v_1 \notin X_k$, then there exists a GD-sequence $S$ of $G$ such that $\cup_{r=j}^{k} Y_r \subseteq \hat{S}$ for some $j \in [k]$.
\end{enumerate}
\end{lemma}

\medskip
\begin{lemma}\label{lemma:chain3}
Let $S$ be a GD-sequence of $G$ of type (c) and let $i\in [k]$ be the largest index such that $\cup_{r=1}^{i} X_r \subseteq \hat{S}$. If $i < k$,  then the following is true:
\begin{enumerate}[(1)]
\item If $i \leq k-2$, then there exists a GD-sequence $S'$ of $G$ such that either $(\cup_{r=i+1}^{k} X_r) \cap \hat{S'} = \emptyset$ or $(\cup_{r=i+2}^{k} X_r) \cap \hat{S'} = \emptyset$ and $|X_{i+1} \cap \hat{S'}| = 1$.
\item If $i = k-1$, then either $X_k \cap \hat{S} = \emptyset$ or $|X_k \cap \hat{S}| = 1$.
\end{enumerate}
\end{lemma}
\begin{proof}
Let $S$ be a GD-sequence of $G$ in which we have a largest index $i$ such that $\cup_{r=1}^{i} X_r \subseteq \hat{S}$. Now, assume that $i < k$ and so, $X_{i+1} \nsubseteq \hat{S}$. If $i=k-1$ and $X_k \cap \hat{S} \neq \emptyset$ then either $|X_k \cap \hat{S}| =1$ or $|X_k \cap \hat{S}| \geq 2$. In the latter case, we get that first vertex of $X_k$ in $S$ footprints itself. So,  Observation~\ref{basicproperty2} ensures that $|X_k \cap \hat{S}| = 1$.

\smallskip
Next, we show that if $i \leq k-2$, then $(\cup_{r=i+2}^{k} X_r) \cap \hat{S} = \emptyset$.  So, let $t \in \{i+2, \ldots, k\}$ be the minimum index such that $X_t \cap \hat{S} \neq \emptyset$. This means that there are some vertices of $\cup_{r=i+1}^{t-1} X_r$ which are not appearing in the sequence $S$. Let $A$ denotes the set of these vertices. Note that $A$ is not the empty set. As $X_t \cap \hat{S} \neq \emptyset$, vertices of $X_t \cap \hat{S}$ appear to footprint some vertices of $N[X_t \cap \hat{S}]= (X_t \cap \hat{S}) \cup (\cup_{r=1}^{t} Y_r)$. Let $x$ be the vertex of $X_t$ which appears first in $S$. We discuss two cases here.

\smallskip
\noindent
\textbf{Case 1:} \underline{$x$ footprints itself}.\\
In this case, we have that no neighbor of $X_t$ appears in $S$ before $x$. So, all vertices in $S$, which footprint vertices of $A$, appear after $x$. Now, we  modify $S$ by removing all such vertices and including all vertices of $A$ in the sequence just after all vertices of $X_t \cap \hat{S}$. We call the modified sequence again by $S$ as it remains a GD-sequence of $G$. Thus, we get a contradiction on $i$ being the largest index satisfying $\cup_{r=1}^{i} X_r \subseteq \hat{S}$. So, this case is not possible.

\smallskip
\noindent
\textbf{Case 2:} \underline{$x$ does not footprint itself}.\\
In this case, we get that all vertices of $X_t$ are footprinted by some vertex of $Y$, which appears before $x$ in $S$. So, $x$ footprints some vertices from the set $\cup_{r=1}^{t} Y_r$. Thus, $|X_t \cap \hat{S}| = 1$ and $X_t \cap \hat{S} = \{x\}$. Now, let $y$ be the vertex which footprints vertices of $A \cap X_{i+1}$ in $S$. Then, there can be two subcases:

\smallskip
\noindent
\textbf{Subcase 2.1:} \underline{$y$ appears after $x$.}\\
In this subcase, we modify $S$ by removing $y$ and including all vertices of $A \cap X_{i+1}$ in the sequence just after $x$. We call the modified sequence again by $S$ as it remains a GD-sequence of $G$. Thus, we get a contradiction on $i$ being the largest index satisfying $\cup_{r=1}^{i} X_r \subseteq \hat{S}$. So, this subcase is not possible.

\smallskip
\noindent
\textbf{Subcase 2.2:} \underline{$y$ appears before $x$.}\\
\textcolor{black}{Here, all vertices of $A$ are footprinted before the appearance of $x$. Recall that $x \in (\cup_{r=i+2}^{k} X_r) \cap \hat{S}$. We get that all vertices of $\cup_{r=i+1}^{k} X_r$ are footprinted before the appearance of $x$. Note that the vertex $x$ itself is footprinted before the appearance of $x$. So, we have, $x$ appears to footprint some vertices of $\cup_{r=1}^{t} Y_r$. This can be further divided in two cases: (i) $x$ does not footprint the vertices of $Y_t$. (ii) $x$ footprints the vertices of $Y_t$.} 

\textcolor{black}{In the first case, $x$ footprints some vertices of the set $\cup_{r=1}^{t-1} Y_r$.  Note that $A \cap X_{t-1} \neq \emptyset$ and vertices of $A \cap X_{t-1}$ do not appear in $S$. Now, we modify $S$ by replacing the vertex $x$ by a vertex of $A \cap X_{t-1}$ and a get a new GD-sequence in which no vertex of $X_t$ appears and one vertex of $X_{t-1}$ appears. If $t=i+2$, then after applying the modification once, we get a GD-sequence $S'$ such that $X_{i+2} \cap \hat{S'} = \emptyset$ and $|X_{i+1} \cap \hat{S'}|=1$. Otherwise, if $t>i+2$, then after applying the modification once, we get a GD-sequence $S'$ such that $X_{t} \cap \hat{S'} = \emptyset$ and $|X_{t-1} \cap \hat{S'}|=1$. Thus, we have another index $t'=t-1 \in \{i+2, \ldots, k\}$ such that  $X_{t'} \cap \hat{S'} \neq \emptyset$.}

\textcolor{black}{In the second case, $x$ footprints vertices of $Y_t$ and we modify $S$ by replacing $x$ with all vertices of $Y_t$. Note that no vertex of $Y_t$ was appearing in the sequence prior to this modification. If $t=i+2$, then after applying the modification once, we get a GD-sequence $S'$ such that $X_{i+2} \cap \hat{S'} = \emptyset$. Otherwise, if $t>i+2$, then after applying the modification once, we get a GD-sequence $S'$ such that $X_{t} \cap \hat{S'} = \emptyset$ and $X_{t-1} \cap \hat{S'}=\emptyset$. Now, we have, either there is no index $t'$ in the set $\{i+2, \ldots, k\}$ such that $X_{t'} \cap \hat{S'} \neq \emptyset$ or there is some $t' \in \{i+2, \ldots, k\}~(t'>t)$ such that  $X_{t'} \cap \hat{S'} \neq \emptyset$.}

\textcolor{black}{In both the cases, we end up with a new GD-sequence $S'$ of $G$. In the former case, we get that $X_t \cap \hat{S'} = \emptyset$ and $|X_{t-1} \cap \hat{S'}| = 1$. The latter case ensures that $X_t \cap \hat{S'} = \emptyset$ and $X_{t-1} \cap \hat{S'} = \emptyset $ (if $t >i+2$). By repeating the above arguments we get that, there is a  GD-sequence $S$ of $G$ such that either $(\cup_{r=i+1}^{k} X_r) \cap \hat{S} = \emptyset$ or $(\cup_{r=i+2}^{k} X_r) \cap \hat{S} = \emptyset$ and $|X_{i+1} \cap \hat{S}| = 1$ for some $i \in [k]$.}
\end{proof}

\smallskip
Analogous to Lemma~\ref{lemma:chain3}, we give a symmetric lemma for the set $Y$ of $G$, whose proof follows similar lines, and is omitted.

\medskip
\begin{lemma}\label{lemma:chain4}
Let $S$ be a GD-sequence of $G$ of type (c) and let $j \in [k]$ be the smallest index such that  $\cup_{r=j}^{k} Y_r \subseteq \hat{S}$. If $j > 1$, then the following is true:
\begin{enumerate}[(1)]
\item If $j \geq 3$, then there exists a GD-sequence $S'$ of $G$ such that either $(\cup_{r=1}^{j-1} Y_r) \cap \hat{S'} = \emptyset$ or $(\cup_{r=1}^{j-2} Y_r) \cap \hat{S'} = \emptyset$ and $|Y_{j-1} \cap \hat{S'}| = 1$.
\item If $j = 2$, then either $Y_1 \cap \hat{S} = \emptyset$ or $|Y_1 \cap \hat{S}| = 1$.
\end{enumerate}
\end{lemma}

\medskip
\begin{lemma}\label{lemma:chain5}
Let $S$ be a GD-sequence of $G$ satisfying all properties of Lemmas~\ref{lemma:chain1}, \ref{lemma:chain2}, \ref{lemma:chain3} and~\ref{lemma:chain4}, let $i\in [k]$ have the role as in Lemma~\ref{lemma:chain3}, and let $j\in [k]$ have the role as in Lemma~\ref{lemma:chain4}. The following statements are true:
\begin{enumerate}[(1)]
\item $j \in \{i, i+1\}$;
\item if $j=i$ then $|X_i|=1$ or $|Y_i|=1$;
\item if $j = i+1$, then either $|X_{i+1} \cap \hat{S}|=1$ or $|Y_{i} \cap \hat{S}|=1$.
\end{enumerate}
\end{lemma}
\begin{proof}
Since $S$ satisfies all properties of Lemmas~\ref{lemma:chain1}, \ref{lemma:chain2}, \ref{lemma:chain3} and~\ref{lemma:chain4}, there are integers $i \in [k-1],j \in \{2, \ldots, k\}$ such that $\cup_{r=1}^{i} X_r \subseteq \hat{S}$ and $\cup_{r=j}^{k} Y_r \subseteq \hat{S}$. It also holds that either $X_{i+1} \cap \hat{S} = \emptyset$ or $|X_{i+1} \cap \hat{S}|=1$. Similarly, either $Y_{j-1} \cap \hat{S} = \emptyset$ or $|Y_{j-1} \cap \hat{S}|=1$. Using Lemmas \ref{lemma:chain3} and \ref{lemma:chain4}, we can also say that if $i \leq k-2$, then  $(\cup_{r=i+2}^{k} X_r) \cap \hat{S} = \emptyset$ and, if  $j \geq 3$, then $(\cup_{r=1}^{j-2} Y_r) \cap \hat{S} = \emptyset$.

\smallskip
First, we show that $j \geq i$. To the contrary, assume that $j < i$. This implies that $X_{i-1} \cup Y_{i-1} \cup X_i \cup Y_i \subseteq \hat{S}$. If $X_i$ appears before $Y_i$, then $Y_{i-1}$ appears after $Y_i$. In this case, we see that $X_{i-1}$ can not appear anywhere in the sequence. So, this case is not possible. If $Y_i$ appears before $X_i$, then $X_{i-1}$ appears after $X_i$. Here,  $Y_{i-1}$ can not appear anywhere in the sequence. So, this case is also not possible. So, we get that $j \geq i$. Now, either $j=i$ or $j \geq i+1$. 

\smallskip
First, we assume $j=i$. If $X_i$ appears before $Y_i$, then an eventual second vertex from $Y_i$ does not footprint any vertex, a contradiction. So, $|Y_i|=1$. In the similar way, we get that $|X_i|=1$, when $Y_i$ appears before $X_i$. Thus, property $(2)$ holds.

\smallskip
Next, assume that $j \geq i+1$. We need to show that $j=i+1$. First, we show that $j \leq i+3$. If $j \geq i+4$, then $X_{i+2} \cap \hat{S} = \emptyset, Y_{i+2} \cap \hat{S} = \emptyset $. This contradicts Observation~\ref{basicproperty3}. So, $j \in \{i+1, i+2, i+3\}$.  If $j=i+3,$ then we see that $\gamma_{gr} (G) \leq \alpha + \beta +2$, where $\alpha = \sum_{r=1}^{i} |X_r|$ and $\beta = \sum_{r=i+3}^{k} |Y_r|$. But, the sequence
$(X_1) \oplus (X_2) \oplus \cdots \oplus (X_i) \oplus (X_{i+1}) \oplus (Y_{k}) \oplus (Y_{k-1}) \oplus \cdots \oplus (Y_{i+3}) \oplus (x) \oplus (X_{i+2})$, where $x \in X_{i+3}$ is a dominating sequence of $G$ of length at least $\alpha + \beta +3$. So, $j \neq i+3$. If $j=i+2$, then we see that $\gamma_{gr} (G) \leq \alpha + \beta' +2$, where $\beta' = \sum_{r=i+2}^{k} |Y_r|$. Now, consider the sequence $ S_0 = (X_1) \oplus (X_2) \oplus \cdots \oplus (X_i) \oplus (Y_{k}) \oplus (Y_{k-1}) \oplus \cdots \oplus (Y_{i+2}) \oplus (y) \oplus (Y_{i+1})$, where $y \in Y_{i}$. If $|Y_{i+1}| > 1$, $S_0$ is a dominating sequence of $G$ having length at least $\alpha + \beta' +3$. This implies that $|Y_{i+1}| = 1$ and $S_0$ is also a GD-sequence of $G$. So, we consider $S_0$ as a GD-sequence of $G$ as it also satisfies Lemmas \ref{lemma:chain1}, \ref{lemma:chain2}, \ref{lemma:chain3} and \ref{lemma:chain4} and thus, $j=i+1$. Thus, property $(1)$ holds.

\smallskip
For the property $(3)$, we show that either $|X_{i+1} \cap \hat{S}|=1$ or $|Y_{i} \cap \hat{S}|=1$. So, first we assume that neither is true, that is, $|X_{i+1} \cap \hat{S}|=0$ and $|Y_{i} \cap \hat{S}|=0$. Now, if $Y_{i+1}$ appears before $X_i$, then the length of $S$ can be increased by including a vertex of $X_{i+1}$ just before $X_i$, but $S$ is a dominating sequence of $G$ of maximum length. So, $Y_{i+1}$ appears after $X_i$, thus length of $S$ can be increased by including a vertex of $Y_i$ just before $Y_{i+1}$, but $S$ is a dominating sequence of $G$ of maximum length. Hence, $|X_{i+1} \cap \hat{S}|=1$ or $|Y_{i} \cap \hat{S}|=1$. If $|X_{i+1} \cap \hat{S}|=1$ and $|Y_{i} \cap \hat{S}|=1$, then suppose that $X_{i+1} \cap \hat{S} = \{a\}$ and $Y_{i} \cap \hat{S} = \{b\}$. \textcolor{black}{Clearly, vertices of the four sets $\{a\}, \{b\}, X_i$ and $Y_{i+1}$ appear in $S$. Let $\mathcal{K} = \{\{a\}, \{b\}, X_i, Y_{i+1}\}$. Recall that all vertices of $X_i$ appear together in $S$. Similarly, all vertices of $Y_{i+1}$ appear together in $S$. Let $A \in \mathcal{K}$ be the set whose vertices appear after the other three sets of $\mathcal{K}$ in the sequence $S$. Then, all vertices of $N[A]$ are footprinted before the appearance of vertices of $A$.} Therefore, either $|X_{i+1} \cap \hat{S}|=1$ or $|Y_{i} \cap \hat{S}|=1$. Thus property $(3)$ holds.
\end{proof}

\medskip
\begin{lemma}\label{lemma:chain6}
Let $S$ be a GD-sequence of $G$ of type (c) and satisfies all properties of Lemmas~\ref{lemma:chain1}, \ref{lemma:chain2}, \ref{lemma:chain3}, \ref{lemma:chain4} and~\ref{lemma:chain5}. Then the sequence $S$ satisfies the following properties
\begin{enumerate}[(1)]
\item If $j=i$ then, $X_{i+1} \cap \hat{S} = \emptyset$ and $Y_{i-1} \cap \hat{S} = \emptyset$;
\item If $j = i+1$, then either $|X_{i+1} \cap \hat{S}|=1$ and $Y_{i} \cap \hat{S} = \emptyset$ or $|Y_{i} \cap \hat{S}|=1$ and $X_{i+1} \cap \hat{S} = \emptyset$.
\end{enumerate}
\end{lemma}
\begin{proof}
First, let $j=i$ and $X_{i+1} \cap \hat{S} \neq \emptyset$. This implies that $|X_{i+1} \cap \hat{S}| = 1$. There are two cases.

\smallskip
\noindent
\textbf{Case 1:} \underline{A vertex of $X_{i+1}$ footprints itself}.\\
In this case, both $Y_i$ and $Y_{i+1}$ appear after $X_{i+1}$. Note that $Y_{i+1}$ appears before $Y_i$. Here, we see that $X_i$ can not appear after $Y_i$ as all vertices in the closed neighborhood of $X_i$ are footprinted before its appearance. So, $X_i$ appears before $Y_i$, but, then all vertices in the closed neighborhood of $Y_i$ are footprinted before its appearance.

\smallskip
\noindent
\textbf{Case 2:} \underline{Vertices of $X_{i+1}$ are footprinted before their appearance in $S$}.\\
Here, at least one of $Y_i$ and $Y_{i+1}$ appear before $X_{i+1}$. So, first we assume that $Y_i$ appears before $X_{i+1}$. Let $A$ denotes the set of vertices which appear before $Y_i$, $B$  denotes the set of vertices which appear before $X_{i+1}$ and after $Y_i$ and $C$ denotes the set of vertices which appear after $X_{i+1}$ in $S$. In this case, $X_i \nsubseteq C$, so $X_i \subseteq A \cup B$. But, then one of $Y_{i+1}$ and $X_{i+1}$ does not footprint any vertex, a contradiction. Similar arguments can be given when $Y_{i+1}$ appears before $X_{i+1}$.
Hence, $X_{i+1} \cap \hat{S} = \emptyset$. In the similar manner, we can prove that $Y_{i-1} \cap \hat{S} = \emptyset$.

\smallskip
Next, we assume that $j = i+1$ and $|X_{i+1} \cap \hat{S}|=1$. We need to show that $Y_{i} \cap \hat{S} = \emptyset$. On the contrary, suppose that $Y_{i} \cap \hat{S} \neq \emptyset$. Here, we see that $S$ contain vertices from all of the sets $X_i, X_{i+1}, Y_i$ and $Y_{i+1}$. Then there exists a vertex $a \in X_i \cup X_{i+1} \cup Y_i \cup Y_{i+1}$ whose closed neighborhood is footprinted before its appearance, a contradiction. In the similar way, we can prove that if $|Y_{i} \cap \hat{S}|=1$, then $X_{i+1} \cap \hat{S} = \emptyset$.
\end{proof}

\medskip
\begin{lemma}\label{lemma:chain7}
Let $S=(v_1, v_2, \ldots, v_p)$ be a GD-sequence of $G$, which satisfies $X \subseteq \hat{S}$. Then one of the following statements is true.\\
$1.$ $\gamma_{gr} (G) = |X|$\\
$2.$ $\gamma_{gr} (G) = |Y|$\\
$3.$ $\gamma_{gr} (G) = |X| + |Y_k|$
\end{lemma}
\begin{proof}
If $S$ is not a type (c) GD-sequence of $G$, then $Y \cap \hat{S} = \emptyset$ and so, $\gamma_{gr} (G) = |X|$.
So, assume that $S$ is a type (c) GD-sequence of $G$.

\smallskip
\noindent
\textbf{Case 1:} \underline{$v_1 \in Y_1$}.\\
In this case, all vertices of $X$ are footprinted by $v_1$. So, all the vertices $v_2, \ldots, v_p$ appear to footprint vertices of $Y \setminus \{v_1\}$ only. This implies that $\gamma_{gr}(G) \leq |Y|$ and so, $\gamma_{gr}(G) = |Y|$.

\smallskip
\noindent
\textbf{Case 2:} \underline{$v_1 \notin Y_1$}.\\
Here, we see that $S$ satisfies all conditons of part (2) of Lemma~\ref{lemma:chain1} with $i=k$. Using Lemma~\ref{lemma:chain4}, we get that there exists an integer $j \in [k]$ such that $\cup_{r=j}^{k} Y_r \subseteq \hat{S}$ and Lemma \ref{lemma:chain5} ensures that $j=k$. Hence, we have that $Y_k \subseteq \hat{S}$ and $(\cup_{r=1}^{k-2} Y_r) \cap \hat{S} = \emptyset$, if $k \geq 3$. Thus, we have that $X \cup Y_k \subseteq \hat{S}$ and so, $\gamma_{gr}(G) \geq |X|+|Y_k|$. Now, if there is no vertex of $Y$ before the vertices of $X_k$ in $S$, then $\gamma_{gr}(G) \leq |X|$. So, all vertices of $Y_k$ appear before vertices of $X_k$.

\smallskip
Next, we have that $Y_{k-1} \cap \widehat{S} = \emptyset$ using Lemma \ref{lemma:chain6}.
Therefore, $\hat{S} = X \cup Y_k$ which implies that $\gamma_{gr} (G) = |X| + |Y_k|$.
\end{proof}

Analogous to Lemma~\ref{lemma:chain7}, we give a symmetric lemma for the set $Y$ of $G$, whose proof follows similar lines, and is omitted.

\medskip
\begin{lemma}\label{lemma:chain8}
Let $S=(v_1, v_2, \ldots, v_p)$ be a GD-sequence of $G$, which satisfies $Y \subseteq \hat{S}$. Then one of the following statements is true.\\
$1.$ $\gamma_{gr} (G) = |Y|$\\
$2.$ $\gamma_{gr} (G) = |X|$\\
$3.$ $\gamma_{gr} (G) = |Y| + |X_1|$
\end{lemma}

\medskip
\begin{lemma}\label{lemma:chain9}
If $G = (X, Y,E)$ is a chain graph such that every GD-sequence of $G$ is of type (c), then for any GD-sequence $S$ of $G$, the following statements are true:
\begin{enumerate}[(1)]
\item  $\cup_{r=1}^{i} X_r \subseteq \hat{S}$ and $\cup_{r=j}^{k} Y_r \subseteq \hat{S}$ for some $i,j \in [k]$.
\item Integers $i$ and $j$ satisfy exactly one of the following:
  \begin{enumerate}
  \item $i<k, j>1$.
  \item $i=1, j=1$.
  \item $i=k, j=k$.
  \end{enumerate}

\item If $i<k, j>1$, then
$\gamma_{gr} (G) = \begin{cases}
  \sum_{r=1}^{i} |X_r| + \sum_{r=j}^{k} |Y_r|  &: j=i \\
  \sum_{r=1}^{i} |X_r| + \sum_{r=j}^{k} |Y_r|+1  &: j=i+1
\end{cases}$

\item If $i=1, j=1$, then $\gamma_{gr} (G) = |Y| + |X_1|$
\item If $i=k, j=k$, then $\gamma_{gr} (G) = |X| + |Y_k|$
\end{enumerate}
\end{lemma}
\begin{proof}
Lemmas~\ref{lemma:chain1} and~\ref{lemma:chain2} ensure property (1).
To prove property (2), we need to show that $i=k, j=1$ cannot be true. So, assume that this is true. Then, Lemma~\ref{lemma:chain5} yields $k=1$, a contradition proving that property (2) holds. Property (3) follows from Lemmas~\ref{lemma:chain3}, \ref{lemma:chain4}, \ref{lemma:chain5} and~\ref{lemma:chain6}. Properties (4) and (5) can be proved using Lemmas~\ref{lemma:chain7} and~\ref{lemma:chain8}.
\end{proof}

We are ready to present an algorithm for computing a GD-sequence of a chain graph based on the above lemmas; see Algorithm~\ref{alg1}.

\begin{algorithm}[h!]
  \label{alg1}
 \caption{\label{algo:grundychain} GD-sequence of a chain graph}
   \textbf{Input:}{ A chain graph $G=(X,Y,E)$ without isolated vertices along with a chain ordering $(x_1, \ldots ,x_{n_1},y_1, \ldots ,y_{n_2})$ of $V(G)$.}\\
  \textbf{Output:}{ A GD-sequence $S$ of $G$.}\\
   \setcounter{AlgoLine}{0}

  \small
   \nl Find the parts $X_1, X_2, \ldots, X_k$ and $Y_1, Y_2, \ldots, Y_k$;

   \nl $i=0$,  $sum[$ $] =0$;
   
   \nl \If{$|Y_{1}|=1$}{
   \nl $sum[0]=n_{2}+|X_{1}|$;}
   \nl \Else{
   \nl $sum[0]=n_{2}$;}

   \nl \For{$i=1:k-1$}{
   \nl $sum[i]=\sum_{j=1}^{i} |X_{j}|+n_{2}-\sum_{j=1}^{i} |Y_{j}|+1$;}

  \nl \If{$|X_{k}|=1$}{
   \nl $sum[k]=n_{1}+|Y_{k}|$;}
   \nl \Else{
   \nl $sum[k]=n_{1}$;}
   
   \nl Find an index $i^*\in \{0,1,2,\ldots,k\}$ for which $sum[i]$ is maximum;

   \nl \If{$i^* == 0$ and $|Y_1|>1$}{

   \nl  $S \leftarrow (Y_{k}) \oplus (Y_{k-1}) \oplus \cdots \oplus (Y_{1})$; }

   \nl \ElseIf{$i^* == 0$ and $|Y_1|=1$}{

   \nl \If{$k \geq 3$}{

   \nl  $S \leftarrow (X_1) \oplus (Y_k) \oplus \cdots \oplus (Y_{3}) \oplus (Y_1) \oplus (Y_2)$; }

   \nl \Else{

   \nl   $S \leftarrow (X_1) \oplus (Y_1) \oplus (Y_2)$; }}

   \nl \ElseIf{$i^* == k$ and $|X_k|>1$}{

   \nl  $S \leftarrow (X_1) \oplus (X_2) \oplus \cdots \oplus (X_k)$;}

 \nl \ElseIf{$i^* == k$ and $|X_k|=1$}{

 \nl \If{$k \geq 3$}{

   \nl  $S \leftarrow (X_1) \oplus (X_2) \oplus \cdots \oplus (X_{k-2}) \oplus (Y_k) \oplus (X_{k}) \oplus (X_{k-1})$;}

   \nl \Else{

   \nl  $S \leftarrow (Y_2) \oplus (X_2) \oplus (X_1)$; }}
   \nl \Else{

   \nl choose a vertex $x \in X_{i^*+1}$

   \nl \If{$i^* >1$}{

   \nl $S \leftarrow (X_1) \oplus (X_2) \oplus \cdots \oplus (X_{i^*-1}) \oplus (Y_{k}) \oplus (Y_{k-1}) \oplus \cdots \oplus (Y_{i^*+1}) \oplus x \oplus (X_{i^*})$;}

   \nl \Else{

   \nl  $S \leftarrow (Y_{k}) \oplus (Y_{k-1}) \oplus \cdots \oplus (Y_2) \oplus x \oplus (X_{1})$; }}

   \nl return $S$.
   \end{algorithm}

\smallskip
 By following the above discussion, note that $\gamma_{gr} (G) \in {\cal A}$, where $${\cal A}=\{n_1, n_2, n_1+|Y_k|, n_2+|X_1|, \sum_{r=1}^{i} |X_r| + \sum_{l=i}^{k} |Y_l|, \sum_{r=1}^{i} |X_r| + \sum_{l=i+1}^{k} |Y_l| +1\},$$ for some $i \in [k-1]$. Thus, the sequence returned by Algorithm \ref{algo:grundychain} is a GD-sequence of $G$. It is easy to see that Algorithm \ref{algo:grundychain} computes $S$ in linear time, {which is the time needed to compute the parts $X_1,\ldots, X_k,Y_1,\ldots,Y_k$}. The following theorem readily follows. 

\medskip
\begin{theorem}\label{thm:chaincorrectness}
Algorithm \ref{algo:grundychain} outputs a GD-sequence $S$ of $G$ in linear time.
\end{theorem}

\medskip
There is a connection between Grundy domination number of a graph and its independence number. 
Let $A$ be an independent set of size $\alpha(G)$. By considering all vertices of $A$ in any order, we get a closed neighborhood sequence of $G$, which yields the well known bound $\gamma_{gr}(G) \geq \alpha(G)$. In addition, for a chain graph $G$, we prove that Grundy domination number is either $\alpha(G)$ or $\alpha(G)+1$.

\medskip
\begin{theorem}
If $G$ is a chain graph, then $\gamma_{gr}(G) \in \{\alpha(G), \alpha(G)+1\}$.
\end{theorem}
\begin{proof}
Let $G$ be a chain graph. By using the notation established in this section, we claim that $\alpha(G) \in \{n_1, n_2, \sum_{j=1}^{i}|X_j|+\sum_{j=i+1}^{k}|Y_j| \text{ for some } i \in [k-1]\}$. To see this, let $A$ be a maximum independent set of $G$. Three cases are possible. If $A \subseteq X$, then $A = X$ implying that $\alpha(G) = n_1$; if $A \subseteq Y$, then $A = Y$ implying that $\alpha(G) = n_2$. Now, if $A \cap X \neq \emptyset$ and $A \cap Y \neq \emptyset$, then one can easily infer that $A = (\cup_{j=1}^{i}X_j) \cup (\cup_{j=i+1}^{k}Y_j)$ for some $i \in [k-1]$ implying that $\alpha(G) = \sum_{j=1}^{i}|X_j|+\sum_{j=i+1}^{k}|Y_j|$. Letting $i_0 \in [k-1]$ be an index such that $\sum_{j=1}^{i_0}|X_j|+\sum_{j=i_0+1}^{k}|Y_j| \geq \sum_{j=1}^{i}|X_j|+\sum_{j=i+1}^{k}|Y_j| \text{ for each } i \in [k-1]$, we may write $\alpha(G) \in \{n_1, n_2, \sum_{j=1}^{i_0}|X_j|+\sum_{j=i_0+1}^{k}|Y_j|\}$. 

\smallskip
Algorithm \ref{algo:grundychain} computes a GD-sequence of $G$ and it turns out that $\gamma_{gr} (G) \in \{n_1, n_2, n_1+|Y_k|, n_2+|X_1|, \sum_{j=1}^{i} |X_j| + \sum_{j=i}^{k} |Y_j|, \sum_{j=1}^{i} |X_j| + \sum_{j=i+1}^{k} |Y_j| +1\}$ for some $i \in [k-1]$. Note that $\gamma_{gr}(G)=n_1 +|Y_k|$ when $|X_k|=1$ implying that $\gamma_{gr}(G) = \sum_{j=1}^{k-1} |X_j| + |Y_k| + 1$. Similarly, $\gamma_{gr}(G)=n_2 +|X_1|$ when $|Y_1|=1$ implying that $\gamma_{gr}(G) =  |X_1| + \sum_{j=2}^{k}|Y_k| + 1$. Now, if $\gamma_{gr}(G)=\sum_{j=1}^{i} |X_j| + \sum_{j=i}^{k} |Y_j|$ for some $i \in [k-1]$ then Lemma \ref{lemma:chain5} ensures that $\gamma_{gr}(G)$ is either $ \sum_{j=1}^{i} |X_j| + \sum_{j=i+1}^{k} |Y_j| + 1$ or $ \sum_{j=1}^{i-1} |X_j| + \sum_{j=i}^{k} |Y_j| + 1$. Hence, $\gamma_{gr}(G) \in \{n_1, n_2, \sum_{j=1}^{i}|X_j|+\sum_{j=i+1}^{k}|Y_j| + 1 \text{ for some } i \in [k-1]\}$. Since Algorithm \ref{algo:grundychain} computes the GD-sequence by finding the maximum of the set $\{n_1, n_2\} \cup \{\sum_{j=1}^{i}|X_j|+\sum_{j=i+1}^{k}|Y_j| + 1: i \in [k-1]\}$, we have, $\gamma_{gr}(G) \in \{n_1, n_2, \sum_{j=1}^{i_0}|X_j|+\sum_{j=i_0+1}^{k}|Y_j|+1\}$. Now, suppose $t= \sum_{j=1}^{i_0}|X_j|+\sum_{j=i_0+1}^{k}|Y_j|$, then we can write that $\alpha(G) \in \{n_1, n_2, t\}$ and $\gamma_{gr}(G) \in \{n_1, n_2, t+1\}$. Now, we consider three cases.

\smallskip
\noindent
\textbf{Case 1}:  $\gamma_{gr}(G) = n_1$:\\
In this case, $n_1 \geq n_2$ and $n_1 \geq t+1 > t$. This implies that $\alpha(G) = n_1$.

\smallskip
\noindent
\textbf{Case 2}:  $\gamma_{gr}(G) = n_2$:\\
In this case, $n_2 \geq n_1$ and $n_2 \geq t+1 > t$. This implies that $\alpha(G) = n_2$.

\smallskip
\noindent
\textbf{Case 3}:  $\gamma_{gr}(G) = t+1$:\\
In this case, if $\alpha(G) = n_1$ then $n_1 \geq t = \gamma_{gr}(G)-1$. So, $\gamma_{gr}(G) \leq \alpha(G)+1$. Similarly, if $\alpha(G) = n_2$ then $n_2 \geq t = \gamma_{gr}(G)-1$. So, $\gamma_{gr}(G) \leq \alpha(G)+1$. Otherwise, $\alpha(G) = t$ implying that 
$\gamma_{gr}(G) = \alpha(G)+1$. 

Therefore, $\gamma_{gr}(G) \in \{\alpha(G), \alpha(G)+1\}$.
\end{proof}

\section{Conclusion}
\label{sec:con}
In this paper, we studied the GDD problem for bipartite graphs and co-bipartite graphs. We proved that the problem is NP-complete for bipartite graphs and efficiently solvable for chain graphs, which form a subclass of bipartite graphs. We also proved NP-completeness of the problem in co-bipartite graphs. To obtain a complete dichotomy, it would be interesting to find the status of the problem in some graph classes that lie between chain graphs and bipartite graphs.

\smallskip
Note that the GDD problem in the class of co-chain graphs (that is, the complements of chain graphs) is easily solvable. Indeed, in a co-chain graph $G=(X,Y,E)$, one can find similar partitions of $X$ and $Y$ into $k$ sets $X_1,\ldots, X_k$, and $Y_1,\ldots,Y_k$, respectively, that arises from the closed twin relation in $G$. Then, one can also immediately infer that $\gamma_{gr}(G)=k$. It would be interesting to see if there are some other known classes of graphs $\cal G$ whose complement class $\overline{{\cal G}}=\{\overline{G}:\, G\in \cal G\}$ has similar status of the computational complexity of the GDD problem as $\cal G$. Two nice instances (with $\cal G$ given by bipartite graphs and chain graphs) are presented in this paper.

\section*{Declarations}
\noindent
\textbf{Conflict of interest}:
The authors have no conflicts of interest to declare that are relevant to the content of this article.

%%%%%%%%%%%%%%%%%%%%%%%%%%%%%
\section*{Acknowledgments}

The first author was supported by the Slovenian Research and Innovation agency (grants P1-0297, J1-2452, J1-3002, and J1-4008).

\end{document}